\documentclass[12pt]{amsart}
\usepackage{amssymb}
\usepackage{amsmath}
\usepackage{amsfonts}
\usepackage{amsthm}
\usepackage{url}
\usepackage[showonlyrefs]{mathtools}
\mathtoolsset{showonlyrefs}
\numberwithin{equation}{section}
\newtheorem{theorem}{Theorem}[section]
\newtheorem{lemma}{Lemma}[section]
\newtheorem{corollary}{Corollary}[section]
\theoremstyle{remark}
\newtheorem{remark}{Remark}[section]
\newtheorem*{ack}{Acknowledgment}
\def\R{\mathbb{R}}
\def\C{\mathbb{C}}

\def\N{\mathbb{N}}

\def\dist{\operatorname{dist}}
\def\re{\operatorname{Re}}
\def\im{\operatorname{Im}}
\def\Per{\operatorname{Per}}
\begin{document}
\title{Lyapunov exponents and related concepts for entire functions}
\author{Walter Bergweiler}
\address{Mathematisches Seminar,
Christian--Albrechts--Universit\"at zu Kiel,
Lude\-wig--Meyn--Str.~4, 24098 Kiel, Germany}
\email{bergweiler@math.uni-kiel.de}
\author{Xiao Yao}
\address{Department of Mathematical Sciences, Tsinghua University, P.\ R.\ China}
\email{yaoxiao0710107@163.com}
\author{Jianhua Zheng}
\address{Department of Mathematical Sciences, Tsinghua University, P.\ R.\ China}
\email{jzheng@math.tsinghua.edu.cn}
\thanks{The second author and the third author were supported by the grant (No. 11571193) of NSF of China}
\begin{abstract}
Let $f$ be an entire function and denote by $f^\#$ be the spherical derivative of~$f$
and by $f^n$ the $n$-th iterate of~$f$.
For an open set $U$ intersecting the Julia set $J(f)$, we consider how fast
$\sup_{z\in U} (f^n)^\#(z)$ and $\int_U (f^n)^\#(z)^2 dx\:dy$ tend to~$\infty$.
We also study the growth rate of the sequence $(f^n)^\#(z)$ for $z\in J(f)$.
\end{abstract}
\maketitle
\section{Introduction and results} \label{intro}
The Julia set $J(f)$ of a rational or entire function~$f$,
which we always assume to be neither constant nor rational of degree~$1$,
is the set of all points where the iterates $f^n$ of $f$ do not form a normal family.  Let
\begin{equation}\label{2a}
f^\#(z)=\frac{|f'(z)|}{1+|f(z)|^2}
\end{equation}
be the spherical derivative of~$f$.
Marty's theorem yields that a point $\xi\in\C$ is contained in $J(f)$ if and and only if
\begin{equation}\label{2a1}
\sup_{n\in\N}\sup_{z\in U} (f^n)^\#(z)=\infty
\end{equation}
for every neighborhood $U$ of $\xi$. Putting
\begin{equation}\label{17e}
\mu(U,f)=\sup_{z\in U} f^\#(z)
\end{equation}
we thus see that the sequence $(\mu(U,f^n))_{n\in\N}$  is unbounded.
It is not difficult to see that it actually tends to $\infty$ and
we are interested in the question how fast it tends to~$\infty$.

Let
\[
M(r,f) = \max_{|z|=r} |f(z)|
\]
be the maximum modulus of $f$ and
denote by $M^n(r,f)$ the iterate of $M(r,f)$ with respect to the first variable; that is,
\begin{equation}\label{2a2}
M^1(r,f)=M(r,f)
\quad\text{and}\quad
M^{n+1}(r,f)=M(M^n(r,f),f).
\end{equation}
It is easy to see that $M^n(R,f)\to\infty$ if $R$ is sufficiently large.
\begin{theorem}\label{thm7}
Let $f$ be an entire function and let $U$ be an open set intersecting the Julia
set of~$f$. Then, for any $R>0$, there exists $m\in \N$ such that
\begin{equation}\label{17f}
\mu(U,f^n)\geq \log M^{n-m}(R,f)
\end{equation}
for large~$n$.
\end{theorem}
For $f(z)=z^d$ we have $\log M^n(R,f)=d^n\log R$ and $\mu(U,f^n)\sim d^n/2$ as $n\to\infty$
if $U$ intersects the unit circle. So Theorem~\ref{thm7} gives the correct order of magnitude
for polynomials.

Next we show that analogous results hold if the supremum of the spherical derivative
is replaced by the normalized spherical area
\begin{equation}\label{2b}
S(U,f)=\frac{1}{\pi}\int_U f^\#(z)^2 dx\:dy.
\end{equation}
For a rational function $f$ of degree $d$ we have
\begin{equation}\label{2e}
c\, d^{n}\leq S(U,f^n)\leq d^{n}.
\end{equation}
for some positive constant $c$ and thus~\cite[Theorem~1]{Yao}
\begin{equation}\label{2e1}
\lim_{n\to\infty} \frac{1}{n}\log S(U,f^n)= \log d .
\end{equation}
Since $\mu(U,f)\geq \sqrt{S(U,f)}$ this implies that
\begin{equation}\label{2e3}
\lim_{n\to\infty} \frac{1}{n}\log \mu(U,f^n)\geq \frac12\log d .
\end{equation}
Barrett and Eremenko \cite[inequality (13) and the remarks following it]{Barrett2014}
showed that we always have strict inequality in~\eqref{2e3}, but that 
the constant $1/2$ on the right hand side
cannot be replaced by a larger constant. 

When dealing with rational functions, it is more systematical to consider
\begin{equation}\label{2c}
\|f'(z)\|=f^\#(z)(1+|z|^2)=|f'(z)|\frac{1+|z|^2}{1+|f(z)|^2}.
\end{equation}
instead of $f^\#(z)$, and this is the quantity considered in~\cite{Barrett2014}.
We note that Theorem~\ref{thm7} holds if $f^\#$ is replaced by
$\|f'\|$ in the definition of $\mu(U,f)$.  An analogous remark applies to the results below.

\begin{theorem}\label{thm5}
Let $f$ be a transcendental entire function and let $U$ be an open set intersecting the Julia
set of~$f$. Then, for any $R>0$, there exists $m\in \N$ such that
\begin{equation}\label{2l}
S(U,f^n)\geq \log M^{n-m}(R,f)
\end{equation}
for large $n$.
\end{theorem}
This result gives the right order of magnitude for the growth of $S(U,f^n)$.
\begin{theorem}\label{thm6}
Let $f$ be a transcendental entire function and let $U$ be a bounded open subset of~$\C$.
Then there exists $R>0$ such that
\begin{equation}\label{17b}
S(U,f^n)\leq \log M^{n}(R,f)
\end{equation}
for all $n\in\N$.
\end{theorem}
We note that it is easy to deduce from Theorems~\ref{thm7} and~\ref{thm5} that
\begin{equation}\label{2e2}
\lim_{n\to\infty} \frac{1}{n}\log \mu(U,f^n)=\infty
\quad\text{and}\quad
\lim_{n\to\infty} \frac{1}{n}\log S(U,f^n)=\infty .
\end{equation}
for a transcendental entire function~$f$.
The second equation answers a question from~\cite{Yao}, where it was shown that this
holds under various additional hypotheses.

We now consider how fast $(f^n)^\#(z)$ can tend to~$\infty$ for a point $z\in J(f)$.
A result of Przytycki says that for rational functions the maximal growth
rate of the sequence $((f^n)^\#(z))$ over all $z\in J(f)$ is essentially the same as the one obtained
when restricting to periodic points $z$ only.
More precisely, Przytycki showed (\cite{Przytycki1994}, the proof is reproduced in~\cite{Gelfert2013}) that
if $f$ is a rational function, then
\begin{equation}\label{2a9a}
\limsup_{n\to\infty}  \frac{1}{n} \sup_{z\in \C} \log  \|(f^n)'(z)\|
=
\sup_{z\in \Per(f)} \lim_{n\to\infty}  \frac{1}{n}\log  \|(f^n)'(z)\|,
\end{equation}
where $\Per(f)$ denotes the set of periodic points of~$f$.
Note that if $z$ is a periodic point of $f$, say $f^p(z)=z$ and $\lambda=(f^p)'(z)$, then
\begin{equation}\label{2a3}
\lim_{n\to\infty}\frac{1}{n}\log \|(f^n)'(z)\|=
\lim_{n\to\infty}\frac{1}{n}\log (f^n)^\#(z)
=\frac{\log|\lambda|}{p}.
\end{equation}
The limit on the right hand side is called the \emph{Lyapunov exponent} of $f$ at $z$ and denoted
by $\chi(f,z)$. More generally,
\begin{equation}\label{2a4}
\overline{\chi}(f,z)=\limsup_{n\to\infty}\frac{1}{n}\log (f^n)^\#(z)
\quad\text{and}\quad
\underline{\chi}(f,z)=\liminf_{n\to\infty}\frac{1}{n}\log (f^n)^\#(z)
\end{equation}
are called the \emph{upper} and \emph{lower Lyapunov exponent} of $f$ at~$z$; see, e.g., \cite{Gelfert2010,Gelfert2013,Levin2015}
for some recent results on Lyapunov exponents for rational maps.

On the left hand side of~\eqref{2a9a} one may replace the supremum over all $z\in\C$ by the supremum over all $z\in U$,
if $U$ is an open set intersecting $J(f)$. Thus~\eqref{2a9a} takes the form
\begin{equation}\label{2a9}
\lim_{n\to\infty} \frac{1}{n}\log \mu(U,f^n)
=
\sup_{z\in \Per(f)} \chi(f,z),
\end{equation}

Eremenko and Levin~\cite[Theorem~3]{Eremenko1990} showed that if $f$ is a polynomial of degree $d\geq 2$, then
there exists a periodic point $z$ such that $ \chi(f,z)\geq \log d$,
with strict inequality unless $f$ is conjugate to the monomial $z\mapsto z^d$.
It follows from~\eqref{2e3} and~\eqref{2a9a} that
if $f$ is a rational function of degree~$d\geq 2$, then
there exists a periodic point $z$ such that $ \chi(f,z)> (\log d)/2$;
see also \cite{Berteloot2010,Gelfert2013,Zdunik2014} for related results.
Finally, \eqref{2e2} and~\eqref{2a9} suggest that if $f$ is a transcendental entire function, then
\begin{equation}\label{2a7}
\sup_{z\in \Per(f)} \chi(f,z)=\infty.
\end{equation}
It follows from the results in~\cite{Bergweiler2005} that this is indeed the case.
\begin{theorem}\label{thm1}
Let $f$ be a transcendental entire function.  Then the set of all $z$
such that $\chi(f,z)=\infty$ is dense in $J(f)$.
\end{theorem}
The essential statement is here that there exists $z\in J(f)$ with $\chi(f,z)=\infty$.
Once this is known, it is easy to see that the set of all such points is dense in $J(f)$.
Note that such points cannot be periodic
since $\chi(f,z)<\infty$ for a periodic point $z$ by~\eqref{2a3}.

It seems plausible that Theorem~\ref{thm1} can be improved by giving a lower bound for
$(f^n)^\#(z)$ which depends on the maximum modulus of~$f$.
However, Theorem~\ref{thm4} below will show that such a lower bound will have to be much
smaller than that given in Theorems~\ref{thm7} and~\ref{thm5}.

We can give such a lower bound for functions in
the Eremenko-Lyubich class $B$ consisting of all transcendental entire functions for which
the set of critical and (finite) asymptotic values is bounded.
In fact, we only need to assume that $f$ has a logarithmic singularity over $\infty$.
This includes functions in $B$ since for such functions all
singularities over $\infty$ are logarithmic.

The lower order $\lambda(f)$ of an entire function $f$ is defined by
\begin{equation}\label{17a}
\lambda(f) = \liminf_{r \to \infty} \frac{\log \log M(r,f)}{\log r}.
\end{equation}
Taking the limes superior in~\eqref{17a} yields the order $\rho(f)$.
\begin{theorem}\label{thm3}
Let $f$ be a transcendental entire function with a logarithmic singularity over $\infty$.
Then the set of all $z$ such that
\begin{equation}\label{5p}
\liminf_{n\to\infty} \frac{1}{n} \log \log (f^n)^\#(z) \geq \log (1+ \lambda(f))
\end{equation}
is dense in $J(f)$.
\end{theorem}
If $f$ has a logarithmic singularity over $\infty$, then $\lambda(f)\geq 1/2$; see, e.g.,
\cite[Proof of Corollary~2]{Bergweiler1995a} or \cite[p.~1788]{Langley1995} for this observation.
Hence Theorem~\ref{thm3} yields the following result.
\begin{corollary}\label{thm2}
Let $f$ be a transcendental entire function with a logarithmic singularity over $\infty$.
Then the set of all $z$ such that
\begin{equation}\label{5p0}
\liminf_{n\to\infty} \frac{1}{n} \log \log (f^n)^\#(z) \geq \log\frac32
\end{equation}
is dense in $J(f)$.
\end{corollary}
Theorems~\ref{thm3} and Corollary~\ref{thm2} are sharp. More precisely, we have the
following result.
\begin{theorem}\label{thm4}
For each $\rho\in [1/2,\infty)$ there exists $f\in B$ with
$\lambda(f)=\rho(f)=\rho$ such that if $z\in \C$ satisfies $\chi(f,z)=\infty$, then
\begin{equation}\label{5p3}
\limsup_{n\to\infty} \frac{1}{n} \log \log (f^n)^\#(z) \leq \log (1+ \rho) .
\end{equation}
\end{theorem}

\begin{ack}
We thank Alexandre Eremenko and Lasse Rempe-Gillen for helpful comments.
\end{ack}

\section{Background from complex dynamics and function theory} \label{prelim}
For an introduction to the iteration theory of entire functions we refer to~\cite{Bergweiler1993,Schleicher2010}.
A basic result of the theory is the following lemma.
\begin{lemma}\label{rpp}
The Julia set of a transcendental entire function is the closure of the set of repelling periodic points.
\end{lemma}
For rational functions this result was obtained by both Fatou and Julia, for transcendental
entire functions it is due to Baker~\cite{Baker1968}.

The \emph{exceptional set} $E(f)$ of an entire function $f$ is the set of all
$z\in\C$ for which the backward orbit
\begin{equation} \label{8a}
O^-(z)=\bigcup_{n=0}^\infty f^{-n}(z)
\end{equation}
is finite. It is a simple consequence of Picard's theorem that $E(f)$ contains at most one point.
The following result is sometimes called the ``blowing-up property'' of the Julia set.
\begin{lemma}\label{excep-set}
Let $f$ be entire, $U\subset \C$ open with $U\cap J(f)\neq\emptyset$ and $K\subset\C\backslash E(f)$ compact.
Then $f^n(U)\supset K$ for all large $n\in\N$.
\end{lemma}

The \emph{escaping set}
\[
I(f) = \{ x \in \R^m \colon f^n(x) \to \infty \},
\]
introduced in~\cite{Eremenko1989}, plays an important role in transcendental dynamics. Its subset
\begin{equation}
\label{fe2}
A(f) = \left\{ z \in \C \colon \text{ there exists } l \in \N\text { with } |f^{n}(z)| > M(R,f^{n-l}) \text{ for } n\geq l\right\},
\end{equation}
where $R> \min _{z \in J(f)}|z|$ and $J(f)$ is the Julia set, is called the
\emph{fast escaping set}. It was introduced in~\cite{Bergweiler1995} and has also turned out
to be very useful in transcendental dynamics.
A thorough study of this set is given in~\cite{Rippon2012} where it is also shown that
\begin{equation}
\label{fe3}
A(f) = \left\{ z \in \C \colon \text{ there exists } l \in \N\text { with } |f^{n}(z)| > M^{n-l}(R,f) \text{ for } n\geq l\right\},
\end{equation}
with $R$ so large that $M^n(R,f)\to\infty$ as $n\to\infty$.
The equivalence of~\eqref{fe2} and~\eqref{fe3} is also apparent from the following lemma proved
in~\cite[Lemma 2.1]{Bergweiler2013}.
\begin{lemma}\label{lem-fe}
Let $f$ be a transcendental entire function and $\varepsilon>0$.
Then there exists $R> 0$ such that if $r > R$ and $n\in \N$, then
\[
 M((1+\varepsilon)r,f^n) \geq M^n(r,f).
\]
\end{lemma}
The following lemma (see, e.g., \cite[Lemma~2.2]{Rippon2009}) is a
a consequence of Hadamard's three circles theorem; that is, the convexity of $\log M(r,f)$ in $\log r$.
\begin{lemma}\label{lem-fe3}
Let $f$ be a transcendental entire function and $c>1$. Then
\[
\log M(r^c,f) \geq c\log M(r,f)
\]
for all  sufficiently large~$r$.
\end{lemma}
The next lemma can be found in~\cite{Eremenko1989} for the escaping  set and
in~\cite{Bergweiler1995,Rippon2012} for the fast escaping  set.
\begin{lemma}\label{JIA}
Let $f$ be entire. Then $J(f)=\partial I(f)=\partial A(f)$.
\end{lemma}
The next lemma consists of Koebe's distortion theorem and Koebe's one quarter theorem.
Here and in the following we denote by $D(a,r)$ the open disk of radius $r$ around~$a$.
\begin{lemma}\label{Koebe}
Let $g\colon D(a,r)\to \C$ be univalent, $0< \rho <1$ and $z\in D(a,\rho r)\backslash\{a\}$.
Then
\begin{equation}\label{ka}
\frac{1}{(1+\rho)^2}\leq \frac{|g(z)-g(a)|}{|g'(a)|\cdot |z-a|}\leq \frac{1}{(1-\rho)^2}
\end{equation}
and
\begin{equation}\label{kb}
\frac{1-\rho}{(1+\rho)^3}\leq \frac{|g'(z)|}{|g'(a)|} \leq \frac{1+\rho}{(1-\rho)^3}.
\end{equation}
Moreover,
\begin{equation}\label{kc}
g(D(a,r))\supset D\!\left(g(a),\frac{1}{4}|g'(a)|r\right).
\end{equation}
\end{lemma}
Koebe's  theorems are usually only stated  for  the special  case that $a=0$,
$r=1$, $g(0)=0$  and $g'(0)=1$, but the above  version follows
immediately from this special case.

The following lemma is Harnack's inequality.
\begin{lemma}\label{Harnack}
Let $u\colon D(a,r)\to \R$ be a positive harmonic function, $0< \rho <1$ and $z\in D(a,\rho r)$.
Then
\begin{equation}\label{harn}
\frac{1-\rho}{1+\rho}\leq \frac{u(z)}{u(a)} \leq \frac{1+\rho}{1-\rho}.
\end{equation}
\end{lemma}

\section{Proof of Theorem~\ref{thm7}} \label{proofthm7}
The following lemma is similar to results given in~\cite{Bergweiler2012,Clunie1966,Pommerenke1970}.
Here and in the following we denote by $D(a,r)$ the open disk around $a$ of radius~$r$.
\begin{lemma}\label{la1}
Let $f\colon D(a,r)\to\C$  be holomorphic and $K,L>0$.
Suppose that $|f(a)|\leq K$ and that $|f'(z)|\leq L$ whenever $|f(z)|= K$. Then
\begin{equation}\label{la1a}
|f(z)|< K \exp\!\left( \frac{2L}{K}|z-a|\right)
\quad\text{for}\ z\in D\!\left(a,\frac{r}{2}\right).
\end{equation}
\end{lemma}
\begin{proof}
We follow the arguments in~\cite[p.~303]{Bergweiler2012}
and  put $u(z)=\log (|f(z)|/K)$ so that $|\nabla u|=|f'/f|$.
With $G=\{z\in D(a,r) \colon |f(z)|>K\}$ we then have $a\notin G$ and
\begin{equation}\label{18a1}
|\nabla u(z)|\leq  \frac{L}{K}
\quad\text{for}\ z\in D(a,r)\cap \partial G.
\end{equation}
Let $z\in G\cap D(a,r/2)$ and put $d(z)=\dist(z,\partial G)$.
Since $a\notin G$ we have $d(z)\leq |z-a|$ and thus
there exists $z_1\in \partial G\cap \partial D(z,d(z))\cap D(a,r)$.
For $0<s<1$ we put $z_s=z+s(z_1-z)=sz_1+(1-s)z$ and deduce from Harnack's inequality that
\begin{equation}\label{18a}
u(z_s)\geq \frac{1-s}{1+s}u(z).
\end{equation}
It follows that
\begin{equation}\label{18b}
\frac{u(z)}{(1+s)d(z)}
\leq
\frac{u(z_s)}{(1-s)d(z)}
=
\frac{u(z_s)-u(z_1)}{(1-s)|z-z_1|}
=
\frac{u(z_s)-u(z_1)}{|z_s-z_1|}.
\end{equation}
Passing to the limit as $s\to 1$ we obtain
\begin{equation}\label{18c}
u(z)\leq 2|\nabla u(z_1)|d(z) \leq \frac{2L}{K} d(z) \leq  \frac{2L}{K}|z-a|
\quad\text{for}\ z\in  G\cap D\!\left(a,\frac{r}{2}\right),
\end{equation}
from which the conclusion follows.
\end{proof}
\begin{proof}[Proof of Theorem~\ref{thm7}]
Since, by Lemma~\ref{rpp}, repelling periodic points are dense in $J(f)$, we may assume
without loss of generality that $U=D(a,r)$ for some repelling periodic point $a$ and some $r>0$.
Since $a$ is periodic we have $|f^n(a)|\leq K$ for some $K$ and all~$n$.
Lemma~\ref{JIA} implies that there exists
$b\in A(f)\cap D(a,r/2)$.  Thus $|f^n(b)| \geq M^{n-m}(R,f)$ for some $m\in\N$ and all $n\geq m$.

With
\begin{equation}\label{17g}
L=\frac{K}{r} \log\frac{M^{n-m}(R,f)}{K}
\end{equation}
we thus have
\begin{equation}\label{17g1}
|f^n(b)|\geq K\exp\!\left(\frac{Lr}{K}\right)
> K\exp\!\left(\frac{2L}{K}|b-a|\right).
\end{equation}
Applying Lemma~\ref{la1} we see that there exists $\xi\in D(0,r)$ with $|(f^n)'(\xi)|>L$ and $|f^n(\xi)|= K$.
It follows that
\begin{equation}\label{17h}
(f^n)^\#(\xi)> \frac{L}{1+K^2}= \frac{K}{r(1+K^2)}  \log\frac{M^{n-m}(R,f)}{K}.
\end{equation}
We may assume here that $r< K/(1+K^2)$ so that the first term on the right side is greater than~$1$.
From this we can deduce that
\begin{equation}\label{17i}
(f^n)^\#(\xi)\geq \log M^{n-m}(R,f)
\end{equation}
for large~$n$.
\end{proof}

\section{Proof of Theorems~\ref{thm5} and~\ref{thm6}} \label{proofthm5}
\begin{proof}[Proof of Theorem~\ref{thm5}]
We may assume that $0$ is a periodic point. Since~\cite[p.~13]{Hayman1964}
the Nevanlinna characteristic $T(r,f)$ and the Ahlfors-Shimizu characteristic $T_0(r,f)$ satisfy
\[
\left|T(r,f)-T_0(r,f)-\log^+|f(0)|\right|\leq \frac12\log 2,
\]
this yields that
\begin{equation}\label{1a}
\left|T(r,f^n)-T_0(r,f^n)\right|\leq C
\end{equation}
for some constant $C$ independent of~$n$.

Choosing $r_1>1$ such that $E(f)\subset D(0,r_1)$ and $r_2>r_1$ we then have
\[
f^k(U)\supset D(0,r_2)\backslash D(0,r_1)
\]
for some $k\in\N$ by Lemma~\ref{excep-set}. Hence
\begin{equation}\label{1d}
S(U,f^{n+k})\geq S(D(0,r_2)\backslash D(0,r_1) ,f^n)= S(r_2,f^n)- S(r_1,f^n),
\end{equation}
with $S(r,f)=S(D(0,r),f)$.

We use the standard estimates
\[
\begin{aligned}
\frac12 S(\sqrt{r},f)\log r
& = \int_{\sqrt{r}}^r\frac{S(\sqrt{r},f)}{t}dt
\leq \int_{0}^r\frac{S(t,f)}{t}dt
\\ &
= T_0(r,f)
\leq S(r,f)\log r+T_0(1,f)
\end{aligned}
\]
which may also be written as
\begin{equation}\label{1b}
\frac{T_0(r,f)-T_0(1,f)}{\log r}
\leq S(r,f)
\leq \frac{T_0(r^2,f)}{\log r}.
\end{equation}
Now \eqref{1d} and \eqref{1b} give
\[
S(U,f^{n+k})\geq
\frac{T_0(r_2,f^n)-T_0(1,f^n)}{\log r_2}-
\frac{T_0(r_1^2,f^n)}{\log r_1}
\geq
\frac{T_0(r_2,f^n)}{\log r_2}-
2\frac{T_0(r_1^2,f^n)}{\log r_1}
\]
and thus
\[
S(U,f^{n+k})\geq
\frac{T(r_2,f^n)}{\log r_2}-
2\frac{T(r_1^2,f^n)}{\log r_1}
-\frac{2C}{\log r_1}
\]
by \eqref{1a}. With the standard estimate
\[
T(r,f)\leq \log^+M(r,f)\leq \frac{R+r}{R-r}T(R,f)
\]
this yields
\begin{equation}\label{1c}
S(U,f^{n+k})\geq
\frac13 \frac{\log M(\frac12 r_2,f^n)}{\log r_2}-
2\frac{\log^+ M(r_1^2,f^n)}{\log r_1}
-\frac{2C}{\log r_1}.
\end{equation}

Lemma~\ref{lem-fe} implies that $M(r_2/2,f^n)\geq M^n(r_2/4,f)$ if $r_2$ was chosen large enough.
Thus \eqref{1c} yields
\begin{equation}\label{1f}
\begin{aligned}
S(U,f^{n+k})
&\geq
\frac13 \frac{\log M^n(\frac14 r_2,f^n)}{\log r_2}-
2\frac{\log^+ M^n(r_1^2,f)}{\log r_1}
-\frac{2C}{\log r_1}
\\ &
\geq
\frac13 \frac{\log M^n(\frac14 r_2,f^n)}{\log r_2}-
3\frac{\log^+ M^n(r_1^2,f)}{\log r_1},
\end{aligned}
\end{equation}
provided $r_1$ is chosen large enough.

With $R_1=M(r_1^2,f)$ and $R_2= M(r_2/4,f)$ this takes the form
\begin{equation}\label{1g}
\begin{aligned}
S(U,f^{n+k})
&\geq
\frac13 \frac{\log M^{n-1}(R_2,f)}{\log R_2}\frac{\log R_2}{\log r_2}-
3\frac{\log^+ M^{n-1}(R_1,f)}{\log R_1}\frac{\log R_1}{\log r_1}.
\end{aligned}
\end{equation}
Since $f$ is transcendental,
\[
\frac{\log M(r,f)}{\log r} \to\infty
\]
as $r\to\infty$. This implies that we may choose $r_1$ and $r_2$ such that
\[
\frac{\log R_1}{\log r_1} \geq 1
\quad\text{and}\quad
\frac{\log R_2}{\log r_2}\geq 12 \frac{\log R_1}{\log r_1}.
\]
We may also assume that $M^{n-1}(R_2,f)\geq 1$ for all~$n$. Thus \eqref{1g} yields
\begin{equation}\label{1h}
\begin{aligned}
S(U,f^{n+k})
&\geq
\left(4 \frac{\log^+ M^{n-1}(R_2,f)}{\log R_2}-3\frac{\log^+ M^{n-1}(R_1,f)}{\log R_1}\right)
\frac{\log R_1}{\log r_1}
\end{aligned}
\end{equation}
if $r_1$ was chosen large enough.

Since $M(r,f)\geq r$ for large~$r$, Lemma~\ref{lem-fe3} yields that
there exists $R_0$ such that for $r\geq R_0$ and $c\geq 1$ we have
\begin{equation}\label{1h1}
\log M^n(r^c,f)\geq c \log M^n(r,f).
\end{equation}
This is equivalent to saying that
\[
\frac{\log M^n(r,f)}{\log r}
\]
is a non-decreasing function of $r$ for $r\geq R_0$.
Since we may assume that  $R_1\geq R_0$  this yields
\[
\frac{\log M^{n-1}(R_2,f)}{\log R_2} \geq \frac{\log M^{n-1}(R_1,f)}{\log R_1}.
\]
It now follows from \eqref{1h} that
\[
S(U,f^{n+k}) \geq \frac{\log M^{n-1}(R_2,f)}{\log R_2}
\]
if $r_2$ (and hence $R_2$) is large enough.

We now choose $l$ such that $M^l(R,f)\geq R^{\log R_2}$ and deduce that
\[
S(U,f^{n+k}) \geq \frac{\log M^{n-1-l}(M^l(R_2,f),f)}{\log R_2}
 \geq \frac{\log M^{n-1-l}(R^{\log R_2},f)}{\log R_2}.
\]
Using~\eqref{1h1} again this finally yields
\[
S(U,f^{n+k}) \geq \log M^{n-1-l}(R,f).
\]
The conclusion now follows with $m=l+k+1$.
\end{proof}

\begin{proof}[Proof of Theorem~\ref{thm6}]
Choose $R\geq e^4$  with
$M^n(R,f)\to\infty$ as $n\to\infty$
such that
$U\subset D\!\left(0,\sqrt{R}\right)$.
Then
\[
\begin{aligned}
S(U,f^n)&\leq S\!\left(\sqrt{R},f^n\right)\leq \frac{2T_0(R,f^n)}{\log R}
\leq  \frac12 \left(T(R,f^n)+C\right)
\\ &
\leq \frac12 \left(\log M(R,f^n)+C\right)\leq \log M^n(R,f)
\end{aligned}
\]
for large $n$ by~\eqref{1a} and~\eqref{1b}. Increasing $R$ if necessary we may achieve that
this holds for all $n\in\N$.
\end{proof}

\section{Proof of Theorem~\ref{thm1}} \label{proofthm1}
We will use the following result \cite[Theorem~1.2]{Bergweiler2005} already quoted in the introduction.
\begin{lemma}\label{multiplier}
Let $f$ be a transcendental entire function and let $p\in\N$, $p\geq 2$.
Then there exists a sequence  $(a_k)$ of fixed points of $f^p$
such that $(f^p)'(a_k)\to\infty$ as $k\to\infty$.
\end{lemma}
\begin{proof}[Proof of Theorem~\ref{thm1}]
We apply this lemma for $p=2$. We may assume that all $a_k$ are repelling fixed points of $f^2$.
Then there exist $r_k>0$ such that $f^2$ is univalent in the disk $D_k=D(a_k,r_k)$. Moreover, we may
assume that there exists an increasing sequence $(\lambda_k)$ tending to $\infty$  such that
\begin{equation}\label{2p1}
\left|(f^2)'(z)\right|\geq \lambda_k>1
\quad\text{for} \  z\in D_k
\end{equation}
and that there exists a domain $W_k$ satisfying $\overline{W_k}\subset D_k$ such that
$f^2\colon W_k\to D_k$ is univalent. This implies that for every $n_k\in\N$ there exists
a domain $V_k$ satisfying $\overline{V_k}\subset D_k$ such that $f^{2n_k}\colon V_k\to D_k$ is univalent.

We put $D_0=U$.  By the Ahlfors islands theorem (see~\cite[Section~5]{Hayman1964} or \cite{Bergweiler1998}),
for each $k\in\N$ there exist $m_k\in\N$ and a subdomain $U_{k-1}$ of $D_{k-1}$ such that
$f^{m_k}\colon U_{k-1}\to D_j$ is univalent for some $j\in\{k,k+1,k+2\}$. We may assume that this
holds for $j=k$ since otherwise we may restrict to a subsequence of $(a_k)$.
Thus $f^{m_k}\colon U_{k-1}\to D_k$ is univalent for $k\in\N$.

We conclude that for each $l\in\N$ there exists subdomain $X_l$ of $U_0$ such that
\begin{equation}\label{2p}
f^{m_l}\circ f^{2n_{l-1}}\circ f^{m_{l-1}}\circ \dots \circ f^{m_2}\circ f^{2n_{1}}\circ f^{m_{1}}
\colon X_l\to D_l
\end{equation}
is univalent, with $\overline{X_{l+1}}\subset X_l$. It follows that there exists
\begin{equation}\label{2q}
z\in\bigcap_{l=1}^\infty X_l =\bigcap_{l=1}^\infty \overline{X_l}.
\end{equation}
We show that we can achieve $\chi(f,z)=\infty$ by choosing the sequence $(n_k)$ rapidly increasing.

In order to do so we note that once the sequences $(m_l)$ and $(U_l)$ are fixed, there
are also sequences $(\alpha_l)$ and $(\beta_l)$ of positive numbers such that
\begin{equation}\label{20c}
\left|(f^k)'(\zeta)\right|\geq \alpha_l
\quad\text{and}\quad
\left|(f^k)(\zeta)\right|\leq \beta_l
\quad\text{for}\ 0\leq k\leq m_l
\ \text{and}\ \zeta\in U_{l-1}.
\end{equation}
Here, as usual, $f^0(\zeta)=\zeta$, so for $k=0$ the first inequality just means that $\alpha_l\leq 1$.

With the sequence $(n_k)$ still to be determined, we define sequences $(N_l)$ and $(M_l)$ by
\begin{equation}\label{20a}
N_l=\sum_{k=1}^l (2n_j+m_j) \quad\text{and}\quad M_l=N_l-2n_l=N_{l-1}+m_l
\end{equation}
so that $M_l<N_l<M_{l+1}$. We may choose $(n_k)$ such that
\begin{equation}\label{20a1}
n_l=\frac12 (N_l-M_{l}) \geq \frac14 (N_l+m_{l+1})+1=\frac14 M_{l+1} +1
\end{equation}
and
\begin{equation}\label{20a2}
\lambda_l^{n_l/2} \geq \frac{1+\beta_{l+1}^2}{\prod_{k=1}^{l+2} \alpha_k}
\geq
\frac{1+\beta_{l+1}^2}{\prod_{k=1}^{l+1} \alpha_k}
\end{equation}
for all~$l$.

Suppose first that $n\in\N$ is such that $N_l<n\leq  M_{l+1}=N_l+m_{l+1}$ for some $l\in\N$.
We deduce from~\eqref{2p1}, \eqref{20c}, \eqref{20a1} and~\eqref{20a2} that
\begin{equation}\label{20b}
\begin{aligned}
\left|(f^n)'(z)\right|
&=
\left|(f^{n-N_l}\circ f^{2n_{l}}\circ f^{m_{l}}\circ \dots \circ f^{m_2}\circ f^{2n_{1}}\circ f^{m_{1}})(z)\right|
\\ &
\geq
\prod_{k=1}^{l}  \lambda_k^{n_k}
\cdot
\prod_{k=1}^{l+1} \alpha_k
\geq
\lambda_l^{n_l} \cdot
\prod_{k=1}^{l+1} \alpha_k
\geq
(1+\beta_{l+1}^2) \lambda_l^{n_l/2}
\\ &
\geq
(1+\beta_{l+1}^2) \lambda_l^{M_{l+1}/8}
\geq
(1+\beta_{l+1}^2) \lambda_l^{n/8}
\end{aligned}
\end{equation}
and hence, using~\eqref{20c} again, that
\begin{equation}\label{20b1}
\left|(f^n)^\#(z)\right|\geq \frac{\left|(f^n)'(z)\right|}{1+\beta_{l+1}^2}\geq \lambda_l^{n/8}.
\end{equation}
Suppose next that $ M_{l+1}<n\leq N_{l+1}=M_{l+1}+2n_{l+1}$ for some $l\in\N$.
Using the same arguments as before we find that
\begin{equation}\label{20b2}
\begin{aligned}
\left|(f^n)'(z)\right|
&=
\left|(f^{n-M_{l+1}}\circ f^{m_{l+1}}\circ f^{2n_{l}}\circ \dots \circ f^{m_2}\circ f^{2n_{1}}\circ f^{m_{1}})(z)\right|
\\ &
\geq \lambda_{l+1}^{(n-M_{l+1})/2} \cdot \prod_{k=1}^{l}  \lambda_k^{n_k}
\cdot \prod_{k=1}^{l+1} \alpha_k
\geq \lambda_{l+1}^{(n-M_{l+1})/2}
(1+\beta_{l+1}^2) \lambda_l^{n_l/2}
\end{aligned}
\end{equation}
if $n-M_{l+1}$ is even while
\begin{equation}\label{20b3}
\begin{aligned}
\left|(f^n)'(z)\right|
&=
\left|(f^{n-M_{l+1}}\circ f^{m_{l+1}}\circ f^{2n_{l}}\circ \dots \circ f^{m_2}\circ f^{2n_{1}}\circ f^{m_{1}})(z)\right|
\\ &
\geq \alpha_{l+2}\cdot \lambda_{l+1}^{(n-M_{l+1}-1)/2} \cdot \prod_{k=1}^{l}  \lambda_k^{n_k}
\cdot \prod_{k=1}^{l+1} \alpha_k
\geq \lambda_{l+1}^{(n-M_{l+1}-1)/2}
(1+\beta_{l+1}^2) \lambda_l^{n_l/2}
\end{aligned}
\end{equation}
if $n-M_{l+1}$ is odd. Thus
\begin{equation}\label{20d}
\left|(f^n)'(z)\right|\geq \lambda_{l+1}^{(n-M_{l+1}-1)/2} (1+\beta_{l+1}^2) \lambda_l^{n_l/2}
\end{equation}
in both cases. Since $(\lambda_k)$ is increasing and
\begin{equation}\label{20d1}
n-M_{l+1}-1+n_l\geq n-\frac34 M_{l+1} \geq \frac{n}{4}
\end{equation}
by~\eqref{20a1}, we find that
\begin{equation}\label{20e}
\left|(f^n)'(z)\right|\geq(1+\beta_{l+1}^2) \lambda_{l}^{(n-M_{l+1}-1+n_l)/2} \geq (1+\beta_{l+1}^2)\lambda_l^{n/8},
\end{equation}
which is the same inequality as~\eqref{20b}.  We conclude that~\eqref{20b1} and hence
\begin{equation}\label{20f}
\frac{1}{n}\log\left|(f^n)^\#(z)\right|\geq \frac{1}{8} \log\lambda_l
\end{equation}
holds for all $n\geq N_1$.
Since $l$ and hence $\lambda_l$ tend to $\infty$ with~$n$, this yields $\chi(f,z)=\infty$.

To prove that the set of all $\zeta$ with $\chi(f,\zeta)=\infty$ is dense in $J(f)$ we note that
if this holds for $\zeta=z$, then it also holds for $\zeta=f^n(z)$ if $n\in\N$. More
generally, it holds  for all $\zeta$ for which there exist $m,n\in\N$ such that
$f^m(\zeta)=f^n(z)$ and $(f^m)'(\zeta)\neq 0$. The set of all such $\zeta$ is easily seen
to be dense in $J(f)$, using the Ahlfors island theorem -- or the simpler result that if
$a_1,a_2,a_3\in\C$ are distinct, then the family of all functions holomorphic in a domain
which have no simple $a_j$-points for all $j$ is normal.
\end{proof}
\begin{remark} \label{rem-dense}
Given a sequence $(V_k)$ of open sets intersecting $J(f)$, one may choose the sequences
$(m_k)$ and $(U_k)$ in the above proof such that $f^{l_k}(U_{k-1})\subset V_k$ for some $l_k\leq m_k$.
Using this it is not difficult to see that one may choose $z$ with the additional property
that the orbit of $z$ is dense in $J(f)$.

Similarly, given any sequence $(c_k)$ of positive real numbers tending to~$\infty$, one may choose
$z$ such that $|f^k(z)|\leq c_k$ for all large~$k$.
This can be achieved by choosing $(n_k)$ large; a similar idea appears in~\cite{Rippon2011}.
\end{remark}

\section{Proof of Theorem~\ref{thm3}} \label{proofthm2}
We recall the logarithmic change of variable for a function
$f$ in the Eremenko-Lyubich class; see \cite[\S 2]{Eremenko1992}. For simplicity we will assume that
all singularities of the inverse are in the unit disk and that $|f(0)|<1$.
The general case can be reduced to this.
Let  $U$ be a logarithmic tract of~$f$, that is, a component of $\{z\in\C\colon |f(z)|>1\}$.
Let $H=\{z\in\C\colon\re z>0\}$ be the right half-plane and $W=\exp^{-1}(U)$.
Then there exists a $2\pi i$-periodic holomorphic function
$F\colon W\to H$ satisfying $\exp F(z)=f(e^z)$, and the restriction of $F$ to
a component of $W$ maps this component biholomorphically onto~$H$.

We call $F$ the function obtained from $f$ by a logarithmic change of variable.
The main tool when working with the Eremenko-Lyubich class is the inequality
\begin{equation}\label{4a0}
|F'(z)|\geq\frac{1}{4\pi}\re F(z)
\quad\text{for}\ z\in W
\end{equation}
obtained by them. We will also need the following lower bound for~$F'$.
\begin{lemma}\label{lemma1a}
Let $F\colon W\to H$ be  the function obtained from a logarithmic change of variable as above.
For $z\in W$ let $z_1\in \partial W$ with $|z_1-z|=\dist(z,\partial W)$. Then
\begin{equation}\label{9a}
|F'(\zeta)|\geq\frac{1}{8\pi}\re F(z)
\end{equation}
for all $\zeta$ in the straight line segment from $z$ to $z_1$.
\end{lemma}
\begin{proof}
Let $G\colon H\to W$ be the branch of the inverse of $F$ with $G(F(\zeta))=\zeta$.
Since $G$ is univalent in $D(F(\zeta),\re F(\zeta))$, Koebe's one quarter theorem yields that
\begin{equation}\label{9a1}
W\supset G(H)\supset G(D(F(\zeta),\re F(\zeta)))\supset D\!\left(\zeta,\frac14 |G'(F(\zeta))|\re F(\zeta)\right)
\end{equation}
and hence
\begin{equation}\label{9a2}
\frac{\re F(\zeta)}{4 |F'(\zeta)|}=
\frac14 |G'(F(\zeta))|\re F(\zeta)\leq \dist(\zeta,\partial W)=|\zeta-z_1|.
\end{equation}
We note that \eqref{4a0} follows from this by noting that $\dist(\zeta,\partial W)\leq\pi$.
In fact, this is the proof of~\eqref{4a0} given in~\cite{Eremenko1992}.

To prove~\eqref{9a}, we write $\zeta=z+s(z_1-z)=sz_1+(1-s)z$ with $0<s<1$  and put $u(z)=\re F(z)$.
Harnack's inequality yields that
\begin{equation}\label{9b}
u(\zeta)\geq \frac{1-s}{1+s}u(z).
\end{equation}
Hence
\begin{equation}\label{9c}
|\zeta-z_1|=(1-s) |z-z_1| \leq 2\frac{u(\zeta)}{u(z)}|z-z_1|
=2|z-z_1|\frac{\re F(\zeta)}{\re F(z)}.
\end{equation}
Together with~\eqref{9a2} this yields
\begin{equation}\label{9d}
\frac{1}{4 |F'(\zeta)|}\leq \frac{2|z-z_1|}{\re F(z)},
\end{equation}
from which the conclusion follows since $|z-z_1|\leq\pi$.
\end{proof}

For $f\in B$ and the function $F\colon W\to H$ obtained
from $f$ by the logarithmic change of variable we
put $\alpha=\inf\{\re z\colon z\in W\}$. As in~\cite[\S 3]{Bergweiler2009} we consider the function
$h\colon(\alpha,\infty)\to (0,\infty)$ defined by
\[
h(x)=\max_{\re z=x} \re F(z) =\max_{y\in\R} \re F(x+iy).
\]
Note that $h$ is increasing by the maximum principle. Moreover,
$h$ is convex by analogy to Hadamard's three circles theorem.
\begin{lemma}\label{lemma1}
For $x>\alpha$ let $z_x\in W$ with
\begin{equation}\label{5a0}
\re z_x=x
\quad\text{and}\quad
\re F(z_x)=h(x).
\end{equation}
Then for each $t\in (0,h(x)]$ there exists $\zeta_t \in D(z_x,\dist(z_x,\partial W))$ such that
\begin{equation}\label{4h1}
\re F(\zeta_t)=t \quad\text{and}\quad |F'(\zeta_t)|\geq \frac{h(x)}{8\pi}.
\end{equation}

Moreover, if $t\geq 8\pi$ and if $U_t$ is the component of
$F^{-1}(D(F(\zeta_t),4\pi))$ that contains~$\zeta_t$, then
\begin{equation}\label{4h0}
|F'(z)| \geq  \frac{h(x)}{96\pi} \quad\text{for}\ z\in U_t
\end{equation}
and
\begin{equation}\label{4j0}
U_t\subset D(\zeta_t,2).
\end{equation}
\end{lemma}
\begin{proof}
It follows from Lemma~\ref{lemma1a} that there exists $\zeta_t\in D(z_x,\dist(z_x,\partial W))$
satisfying~\eqref{4h1}.

Let now $t\geq 8\pi$ and, as in the proof of Lemma~\ref{lemma1a},
let $G\colon H\to W$ be the branch of the inverse of $F$ with $G(F(\zeta_t))=\zeta_t$.
Since $G$ is univalent in $D(F(\zeta_t),\re F(\zeta_t))$, Koebe's distortion theorem implies that
\begin{equation}\label{4g}
|G'(w)| \leq |G'(F(\zeta_t))|\frac{1+4\pi/t}{(1-4\pi/t)^3}\leq 12 |G'(F(\zeta_t))|
\quad\text{for}\ w\in D(F(\zeta_t),4\pi)
\end{equation}
and hence that
\begin{equation}\label{4h}
|F'(z)| \geq  \frac{1}{12}|F'(\zeta_t)|
\quad\text{for}\ z\in U_t ,
\end{equation}
which together with~\eqref{4h1} yields~\eqref{4h0}.
Koebe's distortion theorem and~\eqref{4h1} also yield that if $z\in U_t$, then
\begin{equation}\label{4i}
|z-\zeta_t|
\leq |G'(F(\zeta_t))| \frac{4\pi/t}{(1-4\pi/t)^2}
\leq  |G'(F(\zeta_t))| \frac{16\pi}{t}
=  \frac{16\pi}{|F'(\zeta_t)|t} \leq  \frac{128\pi^2}{h(x)t}.
\end{equation}
Since $8\pi\leq t\leq h(x)$ this yields~\eqref{4j0}.
\end{proof}

\begin{lemma}\label{lemma2}
Let $(x_n)_{n\geq 0}$ be a sequence of positive numbers satisfying
\begin{equation}\label{5a1}
x_n> \max\{\alpha,8\pi\}
\quad\text{and}\quad
x_{n+1}\leq h(x_n)
\end{equation}
for all $n\geq 0$. Then there exists $u\in W$ such that
\begin{equation}\label{5a}
\left|\re F^n(u)-x_n\right|\leq 4\pi
\quad\text{and}\quad
\left|(F^n)'(u)\right|\geq \frac{1}{(96\pi)^n} \prod_{j=0}^{n-1} h(x_j)
\end{equation}
for all $n\geq 1$.
\end{lemma}
\begin{proof}
First we choose $z_0$ with $\re z_0=x_0$ and $\re F(z_0)=h(x_0)$.
By Lemma~\ref{lemma1} there exist $\zeta_0\in D(z_0,2)$ such that
with $\xi_1= F(\zeta_0)$ we have
$\re \xi_1=x_1$ and such that the component $V_1$ of $F^{-1}(D(\xi_1,4\pi)$
that contains $\zeta_0$ satisfies $V_1\subset D(\zeta_0,2)$ and
\begin{equation}\label{5b}
|F'(z)|\geq \frac{1}{96\pi} h(x_0)
\quad\text{for}\ z\in V_1 .
\end{equation}
We now choose a point $z_1$ with $\re z_1=\re \xi_1=x_1$ and $|\im z_1-\im \xi_1|\leq\pi$ such that $\re F(z_1)=h(x_1)$.
Using Lemma~\ref{lemma1} again, we see that there exists $\zeta_1\in D(z_1,\pi)\subset D(\xi_1,2\pi)$ such that
with $\xi_2= F(\zeta_1)$ we have $\re \xi_2=x_2$, and the component $U_2$ of $F^{-1}(D(\xi_2,4\pi)$
that contains $\zeta_1$ satisfies $U_2\subset D(\zeta_1,2)$ and
\begin{equation}\label{5c}
|F'(z)|\geq \frac{1}{96\pi} h(x_1)
\quad\text{for}\ z\in U_2.
\end{equation}
Note that that $U_2\subset D(\xi_1,2+2\pi)$.
Since $F\colon V_1\to D(\xi_1,4\pi)$ is biholomorphic we deduce that
there exists a domain $V_2$ satisfying $\overline{V_2}\subset V_1$ such
that $F\colon V_2\to U_2$ is biholomorphic.
Hence $F^2\colon V_2\to D(\xi_2,4\pi)$ is biholomorphic. Moreover, it follows from~\eqref{5b} and~\eqref{5c} that
\begin{equation}\label{5d}
|(F^2)'(z)|\geq \frac{1}{(96\pi)^2} h(x_0)h(x_1)
\quad\text{for}\ z\in V_2.
\end{equation}
Inductively we thus find a sequence $(\xi_n)$ of points satisfying
$\re \xi_n=x_n$ and a sequence $(V_n)$ of domains satisfying $\overline{V_n}\subset V_{n-1}$ such
that $F^n\colon V_{n}\to D(\xi_n,4\pi)$ is biholomorphic
\begin{equation}\label{5e}
|(F^n)'(z)|\geq \frac{1}{(96\pi)^n} \prod_{j=0}^{n-1}h(x_j)
\quad\text{for}\ z\in V_n.
\end{equation}
The conclusion now follows by choosing $u\in\bigcap_{n=1}^\infty V_n$.
\end{proof}

\begin{proof}[Proof of Theorem~\ref{thm3}]
By hypothesis we have
\begin{equation}\label{5f}
\log M(r,f)\geq r^{\lambda(f)-o(1)}
\end{equation}
as $r\to\infty$. In terms of $F$ this takes the form
\begin{equation}\label{6a}
h(x)\geq e^{(\lambda-\varepsilon(x))x}
\end{equation}
where $\lambda=\lambda(f)$ and $\varepsilon(x)\to 0$.
We may assume here that $\varepsilon$ is non-increasing, since otherwise we may replace it by
\begin{equation}\label{6a1}
\varepsilon^*(x)=\sup_{t\geq x} \varepsilon(t).
\end{equation}
We now consider, for $x>1$,
\begin{equation}\label{6a2}
\delta(x)=\max\left\{\varepsilon(x),\frac{1}{\log x}\right\}.
\end{equation}
Then $\delta$ is non-increasing. We now choose $x_0$ large,
\begin{equation}\label{6a3}
x_{1}=\frac{\lambda-\delta(x_0)}{1+\delta(x_0)}x_0
\quad\text{and}\quad
x_{n+1}=\frac{1+\lambda}{1+\delta(x_n)}x_n
\end{equation}
for $n\geq 1$.
It follows from~\eqref{6a3} that there exists a sequence $(\eta_n)$ tending to $0$ such that
\begin{equation}\label{6f}
x_n= ( 1+\lambda+\eta_n)^n,
\end{equation}
provided $x_0$ was chosen large enough.

Induction shows that
\begin{equation}\label{6b}
\sum_{j=0}^{n-1}(\lambda-\delta(x_j))x_j\geq (1+\delta(x_{n-1}))x_n
\end{equation}
for all $n\geq 1$. Indeed, this holds for $n=1$ by the choice of $x_1$ and assuming
that~\eqref{6b} holds we obtain, using that $\delta(x)$ is non-increasing,
\begin{equation}\label{6c}
\begin{aligned}
\sum_{j=0}^{n}(\lambda-\delta(x_j))x_j
&
=\sum_{j=0}^{n-1}(\lambda-\delta(x_j))x_j + (\lambda-\delta(x_n)) x_n
\\ &
\geq (1+\delta(x_{n-1}))x_n+ (\lambda-\delta(x_n)) x_n
\\ &
= (1+\lambda+\delta(x_{n-1})-\delta(x_n)) x_n
\\ &
\geq
(1+\lambda)x_n
= (1+\delta(x_n))x_{n+1}.
\end{aligned}
\end{equation}
It follows from~\eqref{6f} that Lemma~\ref{lemma2} is applicable if $x_0$ was chosen large enough.
With $u$ as in this lemma we thus have
\begin{equation}\label{6d}
\begin{aligned}
\left|(F^n)'(u)\right|
&\geq \frac{1}{(96\pi)^n} \prod_{j=0}^{n-1} h(x_j)
\\ &
\geq \exp\!\left(\sum_{j=0}^{n-1} (\lambda-\delta(x_j))x_j - n\log\!\left(96\pi\right)\right)
\\ &
\geq \exp\!\left((1+\delta(x_{n-1}))x_n - n\log\!\left(96\pi\right)\right).
\end{aligned}
\end{equation}
Since $\exp F(u)=f(e^u)$ we find, with $z=e^u$, that
\begin{equation}\label{5m}
(F^n)'(u)=z\frac{(f^n)'(z)}{f^n(z)}.
\end{equation}
Since $\re F^n(u)\to\infty$ and thus $|f^n(z)|\to\infty$ by~\eqref{6f} and Lemma~\ref{lemma2}, we have
\begin{equation}\label{5n}
(f^n)^\#(z)\geq \frac{|(f^n)'(z)|}{2|f^n(z)^2|}=\frac{|(F^n)'(u)|}{2|z|\exp(\re F^n(u))}
\end{equation}
for large~$n$. Combined with~\eqref{6a2} and~\eqref{6d} this yields
\begin{equation}\label{6e}
\begin{aligned}
(f^n)^\#(z)
&\geq \exp\!\left((1+\delta(x_{n-1}))x_n - n\log (96\pi)- x_n-4\pi \log(2|z|) \right)
\\ &
= \exp\!\left(\delta(x_{n-1})x_n - n\log (96\pi)-4\pi \log(2|z|) \right)
\\ &
\geq \exp\!\left(\frac{x_n}{\log x_n} - n\log (96\pi)-4\pi \log(2|z|) \right).
\end{aligned}
\end{equation}
For large $n$ we thus have
\begin{equation}\label{6g}
(f^n)^\#(z) \geq \exp\!\left(\frac{x_n}{2 \log x_n} \right)
\end{equation}
and hence
\begin{equation}\label{6h}
\begin{aligned}
\log\log  (f^n)^\#(z)
&\geq \log x_n -\log\log x_n -\log 2
\\ &
= n \log( 1+\lambda+\eta_n) - \log n- \log \log( 1+\lambda+\eta_n) -\log 2,
\end{aligned}
\end{equation}
from which \eqref{5p} follows.

Once it is known that there exists one point $z$ satisfying~\eqref{5p}, it follows as in the proof
of Theorem~\ref{thm1} that the set of all such $z$ is dense in $J(f)$.
\end{proof}

\section{Proof of Theorem~\ref{thm4}} \label{proofthm4}
Mittag-Leffler's function
\begin{equation}\label{7a}
E_\alpha(z)=\sum^\infty_{n=0} \frac{z^n}{\Gamma(\alpha n +1)}
\end{equation}
satisfies $\rho(E_\alpha)=\lambda(E_\alpha)=1/\alpha$.
It was shown in~\cite[Section~4]{Aspenberg2012} that $f$ is in the Eremenko-Lyubich class if $0<\alpha<2$.
Since $E_{2}(z)=\cosh\sqrt{z}$ this also holds for $\alpha=2$.

For $0<\alpha<2$ and $\rho=1/\alpha$ we have (see \cite[p.~ 85]{Goldberg2008}
and~\cite[Section~4]{Aspenberg2012}), for sufficiently small $\delta>0$,
\begin{align}
E_{\alpha}(z) &= \varrho \exp\left(z^{\varrho}\right) + O\!\left(\frac{1}{|z|}\right)
\quad \text{for} \ |\arg (z)| \leq \frac{\alpha\pi}{2} +\delta, \label{7b} \\
E_{\alpha}'(z) &= \varrho^2 z^{\varrho-1} \exp\left(z^{\varrho}\right) + O\!\left(\frac{1}{|z|^2}\right)
\quad \text{for} \ |\arg (z)| \leq \frac{\alpha\pi}{2}+\delta , \label{7c} \\
E_{\alpha}(z) &= O\!\left(\frac{1}{|z|}\right)
\quad \text{for} \ \frac{\alpha\pi}{2}+\delta < | \arg (z) | \leq \pi,  \label{7d} \\
E_{\alpha}'(z) &= O\!\left(\frac{1}{|z|^2}\right)
\quad \text{for} \ \frac{\alpha\pi}{2}+\delta < | \arg (z) | \leq \pi.  \label{7e}
\end{align}
This implies that there exists constants $A$ and $B$ such that
\begin{equation}\label{7f}
|E_\alpha'(z)|\leq A |z|^{\rho-1} |E_\alpha(z)| +B
\end{equation}
for all $z\in\C$. With $C=A+B$ we thus have
\begin{equation}\label{7i}
|E_\alpha'(z)|\leq C |z|^{\rho-1} |E_\alpha(z)|
\quad\text{if}\ |z|\geq 1 \ \text{and}\ |E_\alpha(z)| \geq 1.
\end{equation}
Since $E_{2}(z)=\cosh\sqrt{z}$ the last estimate also holds for $\alpha=2$.

We consider the function
\begin{equation}\label{7g}
f(z)= \eta E_\alpha(z)
\end{equation}
where $0<\eta<1$. Since $E_\alpha\in B$ we have $f\in B$.
By choosing $\eta$ small we can achieve that
\begin{equation}\label{7h}
|f(z)|<1 \quad\text{and}\quad |f'(z)|<1 \quad\text{for}\ |z|\leq 1.
\end{equation}
Moreover, \eqref{7i} implies that
\begin{equation}\label{7i1}
|f'(z)|\leq C |z|^{\rho-1} |f(z)|
\quad\text{if}\ |z|\geq 1 \ \text{and}\ |f(z)| \geq 1.
\end{equation}

Suppose now that $z$ satisfies $\chi(f,z)=\infty$ and put $z_n=f^n(z)$ for $n\geq 0$ so that $z_0=z$.
It follows from~\eqref{7h} that if $|z_N|\leq 1$ for some $N\geq 0$, then
$(f^n)'(z)\to 0$ and hence $(f^n)^\#(z)\to 0$ as $n\to\infty$.
Thus we may assume that $|z_n|\geq 1$ for all $n\geq 0$. Using~\eqref{7i1} we see that
\begin{equation}\label{7j}
|(f^n)'(z)| =\prod_{j=0}^{n-1} \left|f'(z_{j})\right|
\leq C^n\prod_{j=0}^{n-1} |z_{j}|^{\rho-1} |z_{j}|
= \frac{C^n|z_n|}{|z_0|} \prod_{j=0}^{n-1} |z_{j}|^{\rho} .
\end{equation}
Hence
\begin{equation}\label{7k}
(f^n)^\#(z)\leq \frac{|(f^n)'(z)|}{|(f^n)(z)|^2}\leq \frac{C^n}{|z_n|} \prod_{j=0}^{n-1} |z_{j}|^{\rho} .
\end{equation}
If
\begin{equation}\label{7l}
|z_n|> \prod_{j=0}^{n-1} |z_{j}|^{\rho} ,
\end{equation}
then~\eqref{7k} yields that $(f^n)^\#(z)\leq C^n$. Since $\chi(f,z)=\infty$, we deduce that~\eqref{7l}
cannot hold for infinitely many~$n$. Thus there exists $n_0\in\N$ such that
\begin{equation}\label{7n}
|z_n|\leq  \prod_{j=0}^{n-1} |z_{j}|^{\rho}
\quad\text{for}\ n\geq n_0.
\end{equation}
We put $t_n=\log|z_n|$. Then the last inequality takes the form
\begin{equation}\label{7o}
t_n\leq \rho \sum_{j=0}^{n-1} t_{j} \quad\text{for}\ n\geq n_0.
\end{equation}
This implies that
\begin{equation}\label{7p}
\sum_{j=0}^{n} t_{j}= t_n + \sum_{j=0}^{n-1} t_{j} \leq (1+\rho) \sum_{j=0}^{n-1} t_{j}.
\end{equation}
Induction yields that
\begin{equation}\label{7q}
\sum_{j=0}^{n} t_{j} \leq c_0(1+\rho)^n \quad\text{for}\ n\geq n_0-1,
\end{equation}
with
\begin{equation}\label{7r}
c_0= (1+\rho)^{-n_0+1} \sum_{j=0}^{n_0-1} t_{j}.
\end{equation}
Using~\eqref{7k} and~\eqref{7q} we find for $n\geq n_0$ that
\begin{equation}\label{7s}
\begin{aligned}
\log (f^n)^\#(z)
&\leq \log\!\left( C^n \prod_{j=0}^{n-1} |z_{j}|^{\rho} \right)
 = n\log C + \rho \sum_{j=0}^{n-1} t_{j}
\leq n\log C + \rho c_0(1+\rho)^{n-1}.
\end{aligned}
\end{equation}
Hence
\begin{equation}\label{7t}
\log (f^n)^\#(z) =O\!\left( (1+\rho)^{n} \right)
\end{equation}
as $n\to\infty$, which yields~\eqref{5p3}.

\begin{remark} \label{rem-IA}
The proof of Theorem~\ref{thm3} shows that if $f\in B$ or, more generally, if $f$ has a logarithmic
singularity over~$\infty$, then there exists $z\in I(f)$ satisfying~\eqref{5p}. In particular,
there exists $z\in I(f)$ with $\chi(f,z)=\infty$.
We do not know whether this holds for all transcendental entire functions~$f$.

On the other hand, it follows from the proof of Theorem~\ref{thm4} that in general there does
not exist $z\in A(f)$ with $\chi(f,z)=\infty$.  Indeed, it is easily seen that if $f$ is as there,
$z\in A(f)$ and $z_n=f^n(z_n)$, then~\eqref{7l} holds for large~$n$. As shown in the proof
of  Theorem~\ref{thm4} this is incompatible with $\chi(f,z)=\infty$.
\end{remark}

\end{document}